\documentclass{amsart}
\usepackage[english, frenchb]{babel}

\usepackage[fixlanguage]{babelbib}

\usepackage[T1]{fontenc}
\usepackage{amsmath}
\usepackage{amsfonts}
\usepackage{amssymb}

\usepackage{amsmath}
\usepackage{amssymb}
\usepackage{amscd}
\usepackage[usenames,dvipsnames]{xcolor}

\usepackage{txfonts}
\usepackage{hyperref}
\hypersetup{colorlinks=true,citecolor=NavyBlue,linkcolor=Brown,urlcolor=Orange}
\usepackage[all,ps,cmtip]{xy}
\usepackage{xspace}
\usepackage{mathrsfs}
\usepackage{mathtools}
\usepackage{extarrows}
\usepackage{ dsfont }
\usepackage[varg]{pxfonts}
\newtheorem{theorem}{Theorem} 
\newtheorem{thm}[theorem]{Th\'eor\`eme}
\newtheorem{cor}[theorem]{Corollaire}
\newtheorem{lem}[theorem]{Lemme}
\newtheorem{prop}[theorem]{Proposition}

\theoremstyle{definition}

\newtheorem{rem}[theorem]{Remarque}

\theoremstyle{remark}

\newcommand{\mbb}{\mathbb}
\newcommand{\QQ}{\mbb{Q}}
\newcommand{\NN}{\mbb{N}}
\newcommand{\ZZ}{\mbb{Z}}
\newcommand{\CC}{\mbb{C}}

\newcommand{\PP}{\mbb{P}}

\renewcommand{\tilde}{\widetilde}

\DeclareMathOperator\Alb{Alb}
\DeclareMathOperator\alb{alb}
\DeclareMathOperator\alg{alg}
\DeclareMathOperator\Br{Br} 
\DeclareMathOperator\CH{CH} 
\DeclareMathOperator\Chow{Chow} 
\DeclareMathOperator\cl{cl} 
\DeclareMathOperator\codim{codim} 
\DeclareMathOperator\Griff{Griff}  
\DeclareMathOperator\ho{hom}
\DeclareMathOperator\HH{\mathscr{H}}
\DeclareMathOperator\nr{nr}
\DeclareMathOperator\Pic{Pic}
\DeclareMathOperator\tor{tor}
\DeclareMathOperator\Zar{Zar}

\newcommand{\newabstract}[1]{%
  \par\bigskip
  \csname otherlanguage*\endcsname{#1}%
  \csname captions#1\endcsname
  \item[\hskip\labelsep\scshape\abstractname.]
}

\title
[$2$--Cycles sur les hypersurfaces cubiques de dimension $5$]{$2$--Cycles sur les hypersurfaces cubiques de dimension $5$ }

\author{Lie Fu et Zhiyu Tian}
\address{Institut Camille Jordan UMR 5208\\ Universit\'e Claude Bernard Lyon 1\\ 69622 Villeurbanne Cedex, France}
\email{fu@math.univ-lyon1.fr}
\address{
Centre International de la Recherche Math\'ematique de P\'ekin\\
P\'ekin Universit\'e\\
100871, P\'ekin, Chine}
\email{zhiyutian@bicmr.pku.edu.cn}

\date{\today}
\thanks{Lie Fu et Zhiyu Tian sont soutenus par le projet HodgeFun (ANR-16-CE40-0011) et par ``Projet Inter-laboratoire 2017'' financ\'e par F\'ed\'eration de Recherche en Math\'ematiques Rh\^one-Alpes/Auvergne CNRS 3490. Lie Fu est aussi soutenu par le projet ECOVA (ANR-15-CE40-0002) et LABEX MILYON (ANR-10-LABEX-0070). Zhiyu Tian est aussi soutenu par les financements ``Recruitment of Global Experts'', NSFC $\text{N}^\circ$11871155 et $\text{N}^\circ$11831013. }

\begin{document}


\begin{abstract}
On \'etudie les cycles alg\'ebriques de codimension $3$ sur les hypersurfaces cubiques lisses de dimension $5$. Pour une telle hypersurface, on d\'emontre d'une part que son groupe de Griffiths des cycles de codimension $3$ est trivial et d'autre part que l'application d'Abel--Jacobi induit un isomorphisme entre son groupe de Chow des cycles de codimension $3$ alg\'ebriquement equivalents \`a z\'ero et sa jacobienne interm\'ediaire.

\newabstract{english}
We study 2-cycles of a smooth cubic hypersurface of dimension $5$. We show that the Griffiths group of 2-cycles is trivial and the Abel--Jacobi map induces an isomorphism between the Chow group of algebraically trivial $2$-cycles and the intermediate Jacobian.
\end{abstract}


\maketitle



\section{Introduction}\label{sect:Intro}
On s'int\'eresse dans cet article \`a l'\'etude des cycles alg\'ebriques des hypersurfaces cubiques (lisses) de dimension 5 d\'efinie sur le corps $\CC$ des nombres complexes. L'objet central est leurs \emph{anneaux de Chow}, qui consistent en des combinaisons lin\'eaires formelles de sous-sch\'emas ferm\'es int\`egres modulo l'\'equivalence rationnelle, munis d'un produit d'intersection bien d\'efini \cite{MR1644323}. 

Soit $X\subset \PP_{\CC}^{6}$ une telle hypersurface. Ses groupes de Chow sont bien compris sauf en codimension 3\,:
\begin{itemize}
\item $\CH^{0}(X)\cong \ZZ$, engendr\'e par la classe fondamentale de $X$.
\item $\CH^{1}(X)\cong \ZZ$, engendr\'e par la classe de section hyperplane $h:=c_{1}\left(\mathcal{O}_{X}(1)\right)$, gr\^ace au th\'eor\`eme classique de Lefschetz sur les sections hyperplanes.
\item $\CH^{2}(X)\cong \ZZ$, engendr\'e par $h^{2}$. Ceci r\'esulte de la d\'ecomposition de la diagonale de Bloch--Srinivas \cite{MR714776}. Donnons quelques d\'etails\,: $X$ \'etant rationnellement connexe, Th\'eor\`eme 1 \emph{loc.~cit.}~implique que $\CH^{2}(X)_{\hom}=\CH^{2}(X)_{\alg}$ et que le dernier est faiblement repr\'esentable. Comme $H^{3}(X, \ZZ)$ est nul, par \cite[Theorem 1.9]{MR805336}, le repr\'esentant alg\'ebrique de $\CH^{2}(X)_{\alg}$ est z\'ero. Par cons\'equent, $\CH^{2}(X)$ est isomorphe \`a son image par l'application classe de cycle.
 
\item $\CH^{4}(X)\cong \ZZ$, engendr\'e par la classe d'une droite projective contenue dans $X$. D'une part, par \cite{MR1283872}, le groupe de Chow des courbes de $X$ est engendr\'e par les droites\,; d'autre part, toutes les droites sont rationnellement \'equivalentes car la vari\'et\'e de Fano des droites de $X$ est de Fano et donc rationnellement connexe.
\item $\CH^{5}(X)\cong \ZZ$, engendr\'e par la classe d'un point de $X$, car $X$ est unirationnelle.\end{itemize}

Dans le travail pr\'esent, on compl\`ete la liste ci-dessus par le calcul de $\CH^{3}(X)$. Rappelons que l'on dispose d'une filtration $$\CH^{3}(X)\supset \CH^{3}(X)_{\ho}\supset \CH^{3}(X)_{\alg}\supset 0$$
par les sous-groupes des cycles alg\'ebriques \emph{homologiquement, resp. alg\'ebriquement} triviaux (\emph{cf.} \cite{MR2115000}). Regardons les trois gradu\'es associ\'es \`a cette filtration\,:

\begin{itemize}
\item Le quotient $\CH^{3}(X)/\CH^{3}(X)_{\ho}$ est isomorphe, \emph{via} l'application classe, \`a $H^{6}(X,\ZZ)\cong \ZZ$, qui est engendr\'e par la classe d'un plan projectif contenu dans $X$.
\item Le quotient $\CH^{3}(X)_{\ho}/\CH^{3}(X)_{\alg}$ est par d\'efinition le troisi\`eme \emph{groupe de Griffiths} \cite[\S 20.3]{MR1988456}, not\'e $\Griff^{3}(X)$.
\item Pour comprendre $\CH^{3}(X)_{\alg}$, on consid\`ere l'application d'Abel--Jacobi (\emph{cf.} \cite[\S 12.1]{MR1988456})$$\Phi\colon\CH^{3}(X)_{\alg}\to J^{5}(X),$$ o\`u $J^{5}(X):=\frac{H^5(X, \CC)}{F^3H^{5}(X)\oplus H^5(X, \ZZ)}$ est la \emph{jacobienne interm\'ediaire} de $X$, qui est une vari\'et\'e ab\'elienne principalement polaris\'ee de dimension 21.
\end{itemize}

\bigskip
Notre r\'esultat principal est d'identifier les deux derniers gradu\'es et de calculer le groupe $\CH^{3}(X)$\,:

\begin{thm}\label{main}
Soit $X$ une hypersurface cubique lisse de dimension $5$ d\'efinie sur le corps des nombres complexes $\CC$. Alors
\begin{enumerate}
\item $\Griff^3(X)=0$, i.e. $\CH^{3}(X)_{\ho}=\CH^{3}(X)_{\alg}$.
\item L'application d'Abel--Jacobi induit un isomorphisme $\Phi\colon\CH^3(X)_{\alg} \xrightarrow{\cong} J^5(X)(\CC)$.
\end{enumerate}
Par cons\'equent, on a une suite exacte courte scind\'ee 
\begin{equation*}
0\to J^{5}(X)(\CC)\xrightarrow{\Phi^{-1}} \CH^{3}(X)\xrightarrow{\cl} \ZZ\to 0,
\end{equation*}
dont un scindage est fourni par la classe d'un plan projectif contenu dans $X$.
\end{thm}

Ren\'e Mboro \cite[Proposition 0.1 et Theorem 0.9]{mboro2017cubic} avait d\'ej\`a \'etabli que le noyau de $\Phi$ est annul\'e par 18.

\begin{rem}
La partie (1) du th\'eor\`eme \ref{main} est l'analogue du th\'eor\`eme de Bloch--Srinivas \cite[Theorem 1 (ii)]{MR714776} qui dit que pour une vari\'et\'e projective lisse complexe, si le groupe $\CH_{0}$ est support\'e sur un sous-ensemble alg\'ebrique de dimension 2, alors le deuxi\`eme groupe de Griffiths s'annule.

La partie (2) du th\'eor\`eme \ref{main} est l'analogue du th\'eor\`eme de Murre \cite{MR805336} (annonc\'e dans \cite{MR777590}),  bas\'e sur les r\'esultats de Merkurjev--Suslin en K-th\'eorie \cite{MR675529} et Bloch--Ogus sur la conjecture de Gersten \cite{MR0412191}, qui dit que pour une vari\'et\'e projective lisse complexe, si le groupe $\CH_{0}$ est support\'e sur un sous-ensemble alg\'ebrique de dimension 1, alors le groupe de Chow des cycles de codimension 2 alg\'ebriquement triviaux est isomorphe, par l'application d'Abel--Jacobi, \`a la jacobienne interm\'ediaire en degr\'e 3. 

 On remarque aussi que le r\'esultat ci-dessus pour $\CH^{3}(X)$ est compatible, en ignorant sa torsion, avec la conjecture de Bloch--Beilinson--Murre (\emph{cf.} \cite{Murre}, \cite[\S 23.2]{MR1988456}, \cite[\S 11.2]{MR2115000}), car les structures de Hodge $H^{i}(X, \QQ)$, pour $3\leq i\leq 6$, sont de coniveau de Hodge $\geq 2$.
\end{rem}

\bigskip

Un outil important utils\'e dans la preuve du th\'eor\`eme \ref{main} est la \emph{cohomologie non ramifi\'ee} (\emph{cf.} \S \ref{sect:Hnr}), qui est un invariant birationnel pour les vari\'et\'es projectives lisses. Avec l'int\'er\^et potentiel pour le probl\`eme de rationalit\'e des hypersurfaces cubiques de dimension 5, on peut se demander si les groupes de cohomologie non ramifi\'ee fournissent une vraie obstruction dans cette situation. Malheureusement, le r\'esultat suivant, qui est probablement connu par les experts, montre que toutes ces obstructions s'annulent\footnote{ On ignore aussi s'il existe une hypersurface cubique lisse de dimension 5 rationnelle.}. Pour un groupe ab\'elien $A$, notons par $\HH^{i}(A)$ le faisceau pour la topologie de Zariski associ\'e au pr\'efaisceau $U\mapsto H^{i}(U, A)$, o\`u l'ouvert de Zariski $U$ est munie de la topologique classique et $H^{i}(-, A)$ est la cohomologie singuli\`ere \`a coefficients dans $A$.

\begin{thm}\label{main2}
Soit $X$ une hypersurface cubique lisse de dimension 5 d\'efinie sur $\CC$. 
Alors
\begin{enumerate}
\item 
\begin{equation*}
H^{p}\left(X_{\Zar}, \HH^{q}(\ZZ)\right)\cong 
\begin{cases}
H^{2p}(X, \ZZ)\simeq \ZZ & \text{ si } 0\leq p=q\leq 5\\
H^{5}(X, \ZZ) \simeq \ZZ^{\oplus 42} & \text{ si } p=2 \text{ et } q=3\\
0 & \text{ sinon.}
\end{cases}
\end{equation*}
\item De mani\`ere similaire,
\begin{equation*}
H^{p}\left(X_{\Zar}, \HH^{q}(\QQ/\ZZ)\right)\cong 
\begin{cases}
H^{2p}(X, \QQ/\ZZ)\simeq \QQ/\ZZ & \text{ si } 0\leq p=q\leq 5\\
H^{5}(X, \QQ/\ZZ) \simeq (\QQ/\ZZ)^{\oplus 42} & \text{ si } p=2 \text{ et } q=3\\
0 & \text{ sinon.}
\end{cases}
\end{equation*}
En particulier, $H^{i}_{\nr}(X, \QQ/\ZZ):=\bigcup_{n\in \NN^{*}} H^{i}_{\nr}(X, \ZZ/n)=0$ pour tout $i>0$.
\item La suite spectrale de coniveau 
$$E_{1}^{p,q}=\bigoplus_{Y\in X^{(p)}}H^{q-p}\left(\CC(Y), \ZZ(-p)\right)\Longrightarrow H^{p+q}(X, \ZZ)$$ est d\'eg\'en\'er\'ee en $E_{2}$ avec le terme $E_{2}^{p,q}=E^{p,q}_{\infty}=H^{p}\left(X_{\Zar}, \HH^{q}(\ZZ)\right)$, pr\'ecis\'e dans (1).
\end{enumerate}
\end{thm}
La valeur de la plupart des groupes de cohomologie dans le th\'eor\`eme \ref{main2} \'etait connue (\emph{cf.} \cite{CTVoisin2012IHC}). L'ingr\'edient le plus important, en plus de \cite{MR0412191} et \cite{CTVoisin2012IHC}, pour avoir le r\'esultat complet ci-dessus est l'annulation de $H^{4}_{\nr}(X, \QQ/\ZZ)$, qui r\'esultera directement du travail de Kahn--Sujatha \cite{kahn2001unramifiedquadric2}, \emph{cf.} \cite[P.4]{mboro2017cubic} et Lemme \ref{nr4} pour les d\'etails. Ceci r\'epond une question de Voisin \cite[P.144, Question 11.2]{MR3586396}. Certains de nos arguments sont inspir\'es de \cite{MR2987669} et \cite{CTVoisin2012IHC}.

\begin{rem}
Les annulations de cohomologie non ramifi\'ee dans le th\'eor\`eme \ref{main2} laissent ouverte la question de savoir si $H^{i}_{\nr}(X, \QQ/\ZZ)$ est \emph{universellement trivial}, \emph{i.e.} si la fl\`eche 
$$H^{i}(K, \QQ/\ZZ)\to H^{i}_{\nr}(X_{K}/K, \QQ/\ZZ)$$ est un isomorphisme pour toute extension du corps $K/\CC$. Pour $i=0, 1, 2$, c'est un isomorphisme pour des raisons g\'en\'erales. Pour $i=3$, c'est encore un isomorphisme, comme \'etabli par Colliot-Th\'el\`ene dans \cite[Th\'eor\`eme 5.6 (vii)]{MR3404380}.
\end{rem}

\bigskip
Donnons la structure de l'article\,: au \S \ref{sect:Hnr}, apr\`es un rappel de quelques r\'esultats de base en cohomologie non ramifi\'ee dont nous aurons besoin, on montre les annulations de $H^{4}_{\nr}(X, \QQ/\ZZ)$ et $H^{5}_{\nr}(X, \QQ/\ZZ)$\,; le \S \ref{sect:Prepa} contient deux ingr\'edients cruciaux pour la preuve du th\'eor\`eme \ref{main}, \`a savoir, le fait que $H^{5}(X, \ZZ)$ est support\'e en codimension 2 et une d\'ecomposition de la diagonale de $X$\,; les d\'emonstrations de la partie (1) et la partie (2) du th\'eor\`eme \ref{main} sont donn\'ees aux \S \ref{sect:Griff} et \S \ref{sect:IsomAJ} respectivement\,; finalement au \S\ref{sect:SS} on montre le th\'eor\`eme \ref{main2}.

\bigskip
\noindent\textit{Remerciements\,:} L'id\'ee de ce travail est issue du groupe de travail que nous avons coorganis\'e \`a Lyon en 2016-2017. Nous remercions la F\'ed\'eration de Recherche en Math\'ematiques Rh\^one-Alpes-Auvergne (FR 3490) pour le support de financement \emph{Projets Inter-laboratoires}. L'article a \'et\'e pr\'epar\'e durant les deux rencontres \emph{G\'eom\'etrie Alg\'ebrique et G\'eom\'etrie Complexe} au CIRM en janvier et d\'ecembre 2017. On remercie Olivier Benoist, Alberto Collino, Jean-Louis Colliot-Th\'el\`ene et les rapporteurs pour leurs questions et nombreuses suggestions qui am\'eliorent beaucoup cet article. 

\section{Cohomologie non ramifi\'ee}\label{sect:Hnr}
\subsection{Rappels}
Nous renvoyons les lecteurs \`a \cite{MR999316}, \cite{MR1327280}, \cite{CTVoisin2012IHC} pour plus de d\'etails sur la cohomologie non ramifi\'ee. Certains des r\'esultats rappel\'es ci-dessous (le lemme \ref{lemma:SES} par exemple) reposent sur la conjecture de Bloch--Kato, \'etudi\'ee par Merkurjev--Suslin \cite{MR675529}, \cite{MR1062517}, Rost \cite{RostNonpub} et Voevodsky \cite{MR2031199}, montr\'ee par Voevodsky \cite{MR2811603} et Rost (\emph{cf.}~\cite{MR2220090}).

Soit $X$ une vari\'et\'e complexe projective lisse connexe. On consid\`ere l'application `identit\'e'
$$\pi\colon X_{\cl}\to X_{\Zar}$$
qui est un morphisme (continu) entre le site de la topologie classique et celui de la topologie de Zariski.

Pour tout groupe ab\'elien $A$ et tout entier positif $i$, on d\'efinit un faisceau pour la topologie de Zariski sur $X$ par l'image directe d\'eriv\'ee du faisceau constant $A$\,:
\begin{equation*}
\HH^{i}(A):= \HH^{i}_{X}(A):=R^{i}\pi_{*}A.
\end{equation*}
Autrement-dit, $\HH^{i}(A)$ est le faisceau pour la topologie de Zariski associ\'e au pr\'efaisceau $U\mapsto H^{i}(U, A)$, o\`u $U$ est un ouvert de Zariski de $X$, muni de la topologie classique et $H^{i}$ signifie la cohomologie singuli\`ere (Betti).

Par d\'efinition, la $i$-\`eme \emph{cohomologie non ramifi\'ee} de $X$ \`a coefficient dans $A$ est
\begin{equation*}
H^{i}_{\nr}\left(X, A\right):= H^{0}\left(X, \HH^{i}(A)\right).
\end{equation*}

Pour toute sous-vari\'et\'e ferm\'ee int\`egre $Y$ de $X$, on note $H^{j}(\CC(Y), A):=\varinjlim_{U}H^{j}(U, A)$ o\`u $U$ parcourt les ouverts de Zariski non vides de $Y$.
Gr\^ace au travail de Bloch--Ogus sur la conjecture de Gersten \cite{MR0412191}, on a une r\'esolution flasque suivante du faisceau $\HH^{i}(A)$ pour la topologie de Zariski\,:

\begin{equation}\label{eqn:Resolution}
0\to \HH^{i}(A)\to H^{i}(\CC(X), A)\xrightarrow{\partial} \bigoplus_{Y\in X^{(1)}} {i_{Y}}_{*}H^{i-1}(\CC(Y), A(-1))\xrightarrow{\partial} \cdots\xrightarrow{\partial} \bigoplus_{Y\in X^{(i)}} {i_{Y}}_{*}A(-i)\to 0,
\end{equation}
o\`u les fl\`eches $\partial$ sont donn\'ees par le r\'esidu topologique (\cite[P.~417]{MR1988456}) et ${i_{Y}}_{*}$ signifie l'image directe pour un faisceau constant par l'immersion ferm\'ee naturelle de $Y$ vers $X$.

Comme la longueur de la r\'esolution (\ref{eqn:Resolution}) est $i$, on a la cons\'equence suivante\,:
\begin{cor}[{\cite[0.3]{MR0412191}}]\label{cor:Vanishing}
Pour tout $j>i$, $H^{j}(X, \HH^{i}(A))=0$.
\end{cor}

Pour comparer la cohomologie \`a coefficients de torsion et celle \`a coefficients dans $\ZZ$, on a besoin souvent du r\'esultat suivant\,:
\begin{lem}\label{lemma:SES}
Pour tout $i,j\in \NN$ et $n\in \NN^{*}$, on a une suite exacte courte
$$0\to H^{j}(X, \HH^{i}(\ZZ))/n\to H^{j}(X, \HH^{i}(\ZZ/n))\to H^{j+1}(X,\HH^{i}(\ZZ))[n]\to 0.$$
Prenant la limite inductive sur $n$, $$0\to H^{j}(X, \HH^{i}(\ZZ))\otimes \QQ/\ZZ\to H^{j}(X, \HH^{i}(\QQ/\ZZ))\to H^{j+1}(X,\HH^{i}(\ZZ))_{\tor}\to 0.$$
\end{lem}
\begin{proof}
Par \cite[Th\'eor\`eme 3.1]{CTVoisin2012IHC}, on a une suite exacte courte de faisceaux pour la topologie de Zariski sur $X$
$$0\to \HH^{i}(\ZZ)\xrightarrow{\times n}\HH^{i}(\ZZ)\to \HH^{i}(\ZZ/n)\to 0.$$ La conclusion d\'ecoule directement de la suite exacte longue associ\'ee.
\end{proof}

Nous aurons besoin du r\'esultat suivant, d\^u \`a Deligne dans un cadre plus g\'en\'eral, qui compare deux suites spectrales\,:
\begin{prop}[\cite{MR1424593}, \emph{cf.} {\cite[Proposition 6.4]{MR0412191}}, {\cite[Remark 4.19]{MR2723320}}]\label{prop:TwoSS}
Soit $X$ une vari\'et\'e projective lisse complexe connexe.
La suite spectrale de coniveau 
$$E_{1}^{p,q}=\bigoplus_{Y\in X^{(p)}}H^{q-p}\left(\CC(Y), A(-p)\right)\Longrightarrow H^{p+q}(X, A)$$
co\"incide \`a partir de la page $E_{2}$ avec la suite spectrale de Leray associ\'ee au morphisme continu $\pi\colon X_{\cl}\to X_{\Zar}$
$$E_{2}^{p,q}=H^{p}(X, \HH^{q}(A))\Longrightarrow H^{p+q}(X, A).$$
En particulier, la filtration d'aboutissement $N^{\bullet}$ sur $H^{k}(X, A)$ est celle par le \emph{coniveau alg\'ebrique}\,:
$$N^{c}H^{k}(X, A)=\sum_{Z}\ker\left(H^{k}(X,A)\to H^{k}(X\backslash Z, A)\right),$$
o\`u $Z$ parcourt les sous-ensembles alg\'ebriques ferm\'es de $X$ de codimension $\geq c$.
\end{prop}

\subsection{Quelques annulations de cohomologie non ramifi\'ee}\label{subsect:VanishingHnr}
Dans la suite de cet article, $X$ est une hypersurface cubique lisse de dimension $5$ d\'efinie sur $\CC$. On calcule sans difficult\'e par \cite[Proposition 18.2]{MR1988456} les nombres de Hodge de $X$: $h^{i,i}(X)=1$ pour tout $0\leq i\leq 5$, $h^{2,3}(X)=21$ et les autres $h^{p,q}(X)$ sont nuls.

D'apr\`es \cite[Proposition 3.3]{CTVoisin2012IHC}, $H^{>0}_{\nr}(X, \ZZ)$ s'annule comme $X$ est une vari\'et\'e rationnellement connexe. On s'int\'eresse alors aux groupes de cohomologie non ramifi\'ee \`a coefficients de torsion $$H^{*}_{\nr}(X, \QQ/\ZZ):=\bigcup_{n\in \NN^{*}} H^{i}_{\nr}(X, \ZZ/n).$$ Notons que l'union ici a bien un sens car pour tous entiers positifs $n, m$, on a l'injectivit\'e de la fl\`eche $H^{i}(\CC(X), \ZZ/n)\to H^{i}(\CC(X), \ZZ/nm)$ gr\^ace aux th\'eor\`emes de Voevodsky \cite{MR2031199}, \cite{MR2811603}.

Le groupe $H^{*}_{\nr}(X, \QQ/\ZZ)$ \'etaient bien compris sauf en degr\'es 4 et 5\,:
\begin{itemize}
\item $H^{0}_{\nr}(X, \QQ/\ZZ)\cong \QQ/\ZZ$.
\item $H^{1}_{\nr}(X, \QQ/\ZZ)\cong H^{1}(X, \QQ/\ZZ)=0$.
\item $H^{2}_{\nr}(X, \QQ/\ZZ)\cong \Br(X)\cong H^{3}(X, \ZZ)_{\tor}=0$. (\emph{cf.} \cite[Proposition 4.2.3]{MR1327280}). 
\item $H^{3}_{\nr}(X, \QQ/\ZZ)=0$, car d'apr\`es l'interpr\'etation \cite[Th\'eor\`eme 1.1]{CTVoisin2012IHC}, il mesure exactement le d\'efaut de la conjecture de Hodge enti\`ere\,; or $H^{4}(X, \ZZ)$ est bien engendr\'e par la classe alg\'ebrique $h^{2}$, o\`u $h=c_{1}\left(\mathcal{O}_{X}(1)\right)$ la classe d'une section hyperplane.
\item $H^{i}_{\nr}(X, \QQ/\ZZ)=0$ pour tout $i>\dim(X)=5$, car le faisceau $\HH^{i}(\QQ/\ZZ)=0$ par le th\'eor\`eme de Lefschetz pour les vari\'et\'es affines (\cite{MR0177422}).
\end{itemize}

Traitons les deux cas restant\,:

\begin{lem}[{\cite[P.4]{mboro2017cubic}}]\label{nr4}
Le groupe $H^4_{\nr}\left(X, \QQ/\ZZ(4)\right)$ est nul.
\end{lem}
\begin{proof}
On fixe un plan projectif $\Pi$ contenu dans $X$. Alors la projection depuis $\Pi$ fournit une application rationnelle $X\dashrightarrow \PP^{3}$. L'\'eclatement de $X$ le long de $\Pi$ r\'esout le lieu d'ind\'etermination et donne une fibration en quadriques 
$$\pi: \tilde X\to \PP^{3}.$$
Notons $Q$ la fibre g\'en\'erique de $\pi$, qui est une surface quadrique lisse sur le corps des fonctions $K:=\CC(\PP^{3})$.

Par  \cite[Theorem 3]{kahn2001unramifiedquadric2}, pour tout $n\in \NN^{*}$, il y a une application surjective
\begin{equation}\label{eqn:SurjToH4nr}
 H^4\left(K, \mu_n^{\otimes 4}\right) \to H^4_{\nr}\left(Q/K, \mu_n^{\otimes 4}\right),
\end{equation}
o\`u le groupe \`a gauche est la cohomologie galoisienne et le groupe \`a droite est la cohomologie non ramifi\'ee, qui est plus g\'en\'eralement bien d\'efinie pour une vari\'et\'e sur un corps quelconque (\emph{cf.}~\cite[\S 2.2]{CTVoisin2012IHC}).

Observons que $H^{4}\left(K, \mu_n^{\otimes 4}\right)$ est nul car la dimension cohomologique du corps $K$ est $3$ \cite[Chapitre II, \S 4.2, Proposition 11]{MR1324577}. Donc $H^4_{\nr}\left(Q/K, \mu_n^{\otimes 4}\right)=0$ par la surjectivit\'e de \eqref{eqn:SurjToH4nr}.
Comme $H^{i}_{\nr}(-, \QQ/\ZZ)$ est un invariant birationnel des vari\'et\'es projectives lisses \cite{MR999316},  on a 
$H^{4}_{\nr}\left(X, \mu_{n}^{\otimes 4}\right)\cong H^{4}_{\nr}\left(\tilde X, \mu_{n}^{\otimes 4}\right)$, qui est nul car c'est un sous-groupe de $H^4_{\nr}\left(Q/K, \mu_n^{\otimes 4}\right)$.
Finalement, $H^4_{\nr}(X, \QQ/\ZZ(4))$ est l'union de ses sous-groupes $H^{4}_{\nr}\left(X, \mu_{n}^{\otimes 4}\right)$, donc est \'egalement nul.
\end{proof}

\begin{lem}\label{nr5}
Le groupe $H^{5}_{\nr}(X, \QQ/\ZZ(5))$ est nul.
\end{lem}
\begin{proof}
La preuve est similaire \`a celle du lemme \ref{nr4} pr\'ec\'edent\,: on choisit une droite $\Lambda$ contenue dans $X$ et l'\'eclatement $\tilde X$ de $X$ le long de $\Lambda$ admet une structure de fibration en coniques sur $\PP^{4}$, par la projection depuis $\Lambda$\,:
$$\pi: \tilde X\to \PP^{4}.$$
Soit $C$ la conique g\'en\'erique sur le corps des fonctions $K:=\CC(\PP^{4})$. Alors pour tout $n\in \NN^{*}$, on a une application surjective (\emph{cf.} \cite[Proposition A.1]{MR1637963})
$$H^{5}(K, \mu_{n}^{\otimes 5})\to H^{5}_{\nr}(C/K, \mu_{n}^{\otimes 5}).$$ Comme la dimension cohomologique de $K$ est 4, ces deux groupes sont nuls \cite[Chapitre II, \S 4.2, Proposition 11]{MR1324577}. Donc par \cite{MR999316}, le groupe $H^{5}_{\nr}(X, \mu_{n}^{\otimes 5})\cong H^{5}_{\nr}(\tilde X, \mu_{n}^{\otimes 5})$, qui est un sous-groupe de $H^{5}_{\nr}(C/K, \mu_{n}^{\otimes 5})$, est aussi nul.
\end{proof}

\section{Pr\'eparations de la preuve}\label{sect:Prepa}
Rappelons que $X\subset \PP^{6}_{\CC}$ est toujours une hypersurface cubique lisse.
\subsection{Coniveau alg\'ebrique 2}
Comme $H^{5,0}(X)=H^{4,1}(X)=0$ et $H^{3,2}(X)=21$ (\emph{cf.} \cite[Proposition 18.2]{MR1988456}), le coniveau de Hodge de $H^{5}(X, \ZZ)$ est \'egal \`a 2. Nous montrons que son coniveau alg\'ebrique est aussi 2. Plus pr\'ecis\'ement, on a l'observation suivante\,:

\begin{prop}\label{coniveau}
(1). Si  $X$ est g\'en\'erale \footnote{C'est-\`a-dire en dehors d'un sous-ensemble alg\'ebrique ferm\'e propre convenable dans l'espace de modules des hypersurfaces cubiques de dimension 5.}, il existe une famille de plans projectifs contenus dans $X$ param\'etr\'ee par une courbe projective lisse $C$, telle que l'application d'Abel--Jacobi $H^1(C, \ZZ) \to H^5(X, \ZZ)$ est un morphisme surjectif de structures de Hodge.\\
(2). Si $X$ est arbitraire (lisse), alors $H^{5}(X, \ZZ)$ est de coniveau alg\'ebrique 2, \emph{i.e.} il existe un sous-ensemble alg\'ebrique ferm\'e $Z$ de codimension 2 dans $X$, tel que l'application de restriction $H^{5}(X,\ZZ)\to H^{5}(X\backslash Z, \ZZ)$ est nulle.

\end{prop}

\begin{proof}
Soient $F$ le sch\'ema de Fano des plans contenus dans $X$ et $P \to F$ la famille universelle des plans. 
\begin{equation*}
\xymatrix{
P \ar[r]^{q} \ar[d]^{p}& X\\
F
}
\end{equation*}
Si $X$ est g\'en\'erale, alors par le r\'esultat de Collino \cite{collino1986AJcubic5fold}, $F$ est une surface lisse et l'application d'Abel--Jacobi induit un isomorphisme entre la vari\'et\'e d'Albanese de $F$ et la jacobienne interm\'ediaire de $X$\,:
\begin{eqnarray*}
F\to &\Alb(F)& \xrightarrow{\simeq} J^5(X)\\
{[\Pi]}&\mapsto& \Phi(\Pi-\Pi_{0})
\end{eqnarray*}
o\`u $\Pi_{0}$ est un plan projectif contenu dans $X$ fix\'e.
En particulier, par \cite[Th\'eor\`eme 12.17]{MR1988456}, 
l'application $$q_{*}p^{*}\colon H^3(F, \ZZ) \xrightarrow{\cong} H^5(X, \ZZ)$$ est un isomorphisme de structures de Hodge.
On choisit $i\colon C\hookrightarrow F$ un diviseur lisse ample. 
Le th\'eor\`eme de Lefschetz implique que le compos\'e
$$H^1(C, \ZZ) \xrightarrow{i_{*}} H^3(F, \ZZ) \xrightarrow{q_{*}p^{*}}H^5(X, \ZZ)$$ est bien un morphisme surjectif de structures de Hodge.

On passe dans le cadre holomorphe. Pour $X$ arbitraire (lisse), il y a un disque analytique $\mathbb{D}$ de petit rayon dans $\CC$ de centre $0$ et une famille $\mathcal{X}\to \mathbb{D}$ d'hypersurfaces cubiques lisses dont la fibre centrale $\mathcal{X}_0\cong X$, tels que  pour tout $t \in \mathbb{D}\backslash\{0\}$, $\mathcal{X}_t$ est g\'en\'erale au sens ci-dessus. On consid\`ere le sch\'ema de Fano relatif ainsi que la famille universelle\,:
\begin{equation*}
\xymatrix{
 \mathcal{Q}:=\mathcal{P}|_\mathcal{C} \ar@{^{(}->}[r]\ar[d]\ar@{}[dr]|\square & \mathcal{P} \ar[r]^{q} \ar[d]^{p}& \mathcal{X}\ar[ddl]\\
 \mathcal{C} \ar@{^{(}->}[r]^{i} & \mathcal{F} \ar[d]& \\
 &\mathbb{D} &
}
\end{equation*}
o\`u $\mathcal{C}$ est un diviseur ample g\'en\'eral de $\mathcal{F}$.
Donc $\mathcal{C} \to \mathbb{D}$ est une famille de courbes qui param\`etrent  des plans dans les fibres de $\mathcal{X}\to \mathbb{D}$. On a alors le diagramme commutatif suivant\,:
\begin{equation*}
\xymatrix{
 & H_{5}(\mathcal{Q}_{t}, \ZZ) \ar[r]^{{q_{t}}_{*}} \ar[d]& H_{5}(\mathcal{X}_{t}, \ZZ) \ar[d]_{\simeq}\\
& H_{5}(\mathcal{Q}, \ZZ) \ar[r]^{q_{*}} & H_{5}(\mathcal{X}, \ZZ)\\
& H_{5}(\mathcal{Q}_{0}, \ZZ) \ar[r]^{{q_{0}}_{*}}  \ar[u]^{\simeq} & H_{5}(\mathcal{X}_{0}, \ZZ) \ar[u]^{\simeq}
}
\end{equation*}
o\`u $t\in \mathbb{D}\backslash\{0\}$, les fl\`eches verticales sont induites par des inclusions naturelles. 

Les fl\`eches verticales entre la deuxi\`eme et la troisi\`eme ligne sont des isomorphismes car la fibre centrale  $\mathcal{Q}_{0}$ (\emph{resp.} $\mathcal{X}_{0}$) est un r\'etract par d\'eformation de l'espace total $\mathcal{Q}$ (\emph{resp.} $\mathcal{X}$). L'autre fl\`eche \`a droite est un isomorphisme car toutes les hypersurfaces cubiques lisses sont diff\'eomorphes et la projection $\mathcal{X}\to \mathbb{D}$ est une submersion par le th\'eor\`eme de fibration d'Ehresmann (\emph{cf.}~\cite[Proposition 9.3]{MR1988456}).
Par le paragraphe pr\'ec\'edent, la fl\`eche de la premi\`ere ligne est la compos\'ee de
\[
H_5(\mathcal{Q}_t, \ZZ) \cong H^1(\mathcal{C}_t, \ZZ) \to H^3(\mathcal{F}_t, \ZZ)\to H^5(\mathcal{X}_t, \ZZ)\cong H^5(\mathcal{X}_t, \ZZ),
\]
donc surjective. 
Il en est donc de m\^eme pour la deuxi\`eme et la troisi\`eme ligne. En particulier, 
$${q_{0}}_{*}: H_{5}(\mathcal{Q}_0, \ZZ) \to H_{5}(\mathcal{X}_{0}, \ZZ)$$
est surjective. De mani\`ere \'equivalente, $H^{5}(\mathcal{X}_{0}, \ZZ)$ est support\'ee sur la pr\'eimage de $\mathcal{C}_{0}$ dans $\mathcal{P}_{0}$, qui est un sous-ensemble alg\'ebrique de dimension 3.
\end{proof}

 \subsection{D\'ecomposition de la diagonale}
 La technique de la \emph{d\'ecomposition de la diagonale} est initi\'ee par Bloch--Srinivas \cite{MR714776}. Nous avons besoin de la version raffin\'ee suivante due \`a Paranjape \cite{MR1283872} et Laterveer \cite{MR1669995}, \emph{cf.} \cite[\S 22]{MR1988456} et \cite[Theorem 3.18]{MR3186044}. 
 
 \begin{thm}[{\cite[Theorem 3.18]{MR3186044}}]\label{thm:DD}
 Soient $Y$ une vari\'et\'e  complexe projective lisse connexe de dimension $n$ et $c$ un entier positif.  Suppose que pour tout $0\leq k<c$, l'application de classe $$\cl\colon\CH_{k}(Y)_{\QQ}\hookrightarrow H^{2n-2k}(Y, \QQ)$$ est injective. Alors il existe un entier non nul $m$ et une d\'ecomposition dans $\CH^{n}(Y\times Y)$
 \begin{equation}\label{eqn:DDgeneral}
  m\Delta_{Y}=Z_{0}+Z_{1}+\cdots+Z_{c-1}+Z',
 \end{equation}
 tels que $Z_{i}$ est support\'e sur $W_{i}\times W_{i}'$ avec $\dim W_{i}=i$ et $\dim W_{i}'=n-i$, et $Z'$ est support\'e sur $Y\times T$ avec $T$ un sous-ensemble alg\'ebrique ferm\'e de codimension $\geq c$ dans $Y$.
 \end{thm}

 Pour les hypersurfaces cubiques de dimension 5, le th\'eor\`eme g\'en\'eral ci-dessus donne le r\'esultat suivant\,:
 
\begin{cor}\label{cor:DDcubique5}
Soit $X\subset \PP^{6}_{\CC}$ une hypersurface cubique lisse. Alors il existe un entier non nul $m$ et une d\'ecomposition  dans $\CH^{5}(X\times X)$
\begin{equation}\label{eqn:DD}
m\Delta_{X}=m~x\times X+m~ l\times h+Z'
\end{equation}
o\`u $x$ est un point arbitraire de $X$, $l$ est une droite projective contenue dans $X$ arbitraire, $h=c_{1}\left(\mathcal{O}_{X}(1)\right)$ est la classe de section hyperplane et $Z'$ est support\'e sur $X\times T$ avec $T$ un sous-ensemble alg\'ebrique ferm\'e de codimension $\geq 2$ dans $X$.
\end{cor}
\begin{proof}
L'hypoth\`ese du th\'eor\`eme \ref{thm:DD} est v\'erifi\'ee pour $c=2$ car $\CH_{0}(X)$ et $\CH_{1}(X)$ sont des groupes cycliques infinis engendr\'es respectivement par un point et une droite dans $X$ (\emph{cf.} \S \ref{sect:Intro} Introduction). On a donc une d\'ecomposition (\ref{eqn:DDgeneral}) avec $c=2$. Pour d\'eterminer $Z_{0}$ et $Z_{1}$, on utilise encore une fois les descriptions de $\CH_{0}(X)$ et $\CH_{1}(X)$ plus le fait que $\CH^{1}(X)\cong \ZZ\cdot h$, pour obtenir que $Z_{0}=m_{1}~x\times X$ et $Z_{1}=m_{2}~ l\times h$ avec deux entiers $m_{1}, m_{2}$. Finalement, pour voir que $m_{1}=m_{2}=m$, on fait agir (\ref{eqn:DDgeneral}) par correspondances sur la classe d'un point $x$ et la classe d'une droite $l$. D'une part, $\Delta_{X}$ agit toujours par l'identit\'e et d'autre part, $(x\times X)^{*}(x)=x$, $(x\times X)^{*}(l)=0$, $(l\times h)^{*}(x)=0$, $(l\times h)^{*}(l)=l$ et ${Z'}^{*}(x)={Z'}^{*}(l)=0$ par raison de dimension, on obtient
$m~ x=m_{1}~ x$ et $ m~ l=m_{2}~ l$. D'o\`u, $m=m_{1}=m_{2}$.
\end{proof}

\section{Preuve de l'annulation de $\Griff^{3}(X)$}\label{sect:Griff}
On d\'emontre dans cette section le th\'eor\`eme \ref{main} (1). Commen\c cons par la r\'einterpr\'etation cohomologique suivante du groupe de Griffiths des cycles de codimension 3. On utilise les notations du \S \ref{sect:Hnr}.

\begin{prop}\label{prop:Griffiths}
Soit $X$ une hypersurface cubique lisse dans $\PP^{6}_{\CC}$.
On a un isomorphisme $H^1(X, \HH^4(\ZZ(4))) \cong \Griff^3(X)$.
\end{prop}
\begin{proof}
 Consid\'erons la suite spectrale de Leray associ\'ee au morphisme continu de sites $\pi: X_{\cl}\to X_{\Zar}$\,:
$$E_2^{p, q}=H^p\left(X, \HH^q(\ZZ)\right) \Rightarrow H^{p+q}(X, \ZZ).$$ 
D'une part, rappelons les annulations du corollaire \ref{cor:Vanishing}: $H^{p}\left(X, \HH^{q}(\ZZ)\right)=0$ pour tout $p>q$\,; d'autre part, d'apr\`es \cite[Proposition 3.3 (i)]{CTVoisin2012IHC}, on a $H^0\left(X, \HH^q(\ZZ)\right)=0$ pour tout $q>0$. En regardant la suite spectrale, on obtient une suite exacte
\begin{equation}\label{eqn:Exact}
0\to H^{2}\left(X, \HH^{3}(\ZZ)\right)\to H^{5}(X, \ZZ)\to H^{1}\left(X, \HH^{4}(\ZZ)\right)\xrightarrow{d_{2}} H^{3}\left(X, \HH^{3}(\ZZ)\right)\to H^{6}(X, \ZZ)
\end{equation}
Gr\^ace \`a la  proposition \ref{prop:TwoSS}, la filtration d'aboutissement est celle de coniveau alg\'ebrique, donc l'image de la premi\`ere fl\`eche $H^2\left(X, \HH^3(\ZZ)\right) \to H^5(X, \ZZ)$ est 
$$N^2H^5(X, \ZZ):=\sum_{\substack{Z \subset X \text{ Zariski ferm\'e}\\ \codim Z\geq 2}}\ker\left(H^{5}(X,\ZZ)\to H^{5}(X\backslash Z, \ZZ)\right),$$
qui est simplement $H^{5}(X, \ZZ)$,  par la proposition \ref{coniveau}. Autrement dit, le morphisme $H^2(X, \HH^3(\ZZ)) \to H^5(X, \ZZ)$ est surjectif (en fait un isomorphisme). 

Par cons\'equent, la suite exacte (\ref{eqn:Exact}) se simplifie et implique que $d_{2}$ induit un isomorphisme
$$H^{1}\left(X, \HH^{4}(\ZZ)\right)\simeq \ker\left(H^{3}\left(X, \HH^{3}(\ZZ)\right)\to H^{6}(X, \ZZ)\right).$$
Or \cite[Corollary 7.4]{MR0412191} \'etablit un isomorphisme $H^{3}\left(X, \HH^{3}(\ZZ)\right)\cong \CH^{3}(X)/\alg$, le groupe des cycles de codimension 3 modulo l'\'equivalence alg\'ebrique. De plus, l'application $H^{3}\left(X, \HH^{3}(\ZZ)\right)\to H^{6}(X, \ZZ)$ s'identifie, \emph{via} cet isomorphisme, \`a l'application classe de cycles, dont le noyau est $\Griff^{3}(X)$ par d\'efinition. On en d\'eduit l'isomorphisme cherch\'e.
\end{proof}

\begin{cor}\label{cor:SansTorsion}
Le troisi\`eme groupe de Griffiths $\Griff^3(X)$ est sans torsion.
\end{cor}
\begin{proof}
Par le lemme \ref{lemma:SES} on a une suite exacte courte
$$0\to H^{4}_{\nr}(X, \ZZ(4))/n\to H^{4}_{\nr}(X, \mu_n^{\otimes 4})\to H^{1}(X, \HH^{4}(\ZZ(4)))[n]\to 0.$$
Comme $H^{4}_{\nr}(X, \ZZ(4))=0$ par \cite[Proposition 3.3 (i)]{CTVoisin2012IHC}, on a que pour tout $n\in \NN^{*}$, $$H^{4}_{\nr}(X, \mu_n^{\otimes 4})\simeq H^{1}(X, \HH^{4}(\ZZ(4)))[n].$$
Il en r\'esulte que $H^4_{\nr}(X, \QQ/\ZZ(4))$ est le sous-groupe de torsion du groupe $H^1(X, \HH^4(\ZZ(4)))$, qui s'identifie \`a celui du groupe $ \Griff^3(X)$ par la proposition \ref{prop:Griffiths}. Donc le lemme \ref{nr4} nous permet de conclure. 
\end{proof}

\begin{rem}
Les arguments dans la proposition \ref{prop:Griffiths} et le corollaire \ref{cor:SansTorsion} sont essentiellement dus \`a Voisin \cite[Corollary 1.3]{MR2987669}.
\end{rem}

\begin{prop}\label{prop:DeTorsion}
Le troisi\`eme groupe de Griffiths $\Griff^3(X)$ est de torsion.
\end{prop}
\begin{proof}
On utilise la d\'ecomposition raffin\'ee de la diagonale (\ref{eqn:DD}) dans le corollaire \ref{cor:DDcubique5}. On peut supposer que $T$ est de dimension pure 3, en g\'en\'eral r\'eductible. Soient $\tilde T\to T$ une d\'esingularisation et $\tilde Z'\in \CH_{5}(X\times \tilde T)$ tel que son image dans $\CH_{5}(X\times X)$ est $m'Z'$ avec certain entier $m'$ . Notons $\tilde i: \tilde T\to X$ le compos\'e de la d\'esingularisation et l'inclusion. On \'ecrit $\tilde T=\coprod_{j} \tilde T_{j}$ la d\'ecomposition en composantes connexes.

Soit $\alpha\in \CH^{3}(X)_{\ho}$ un 2-cycle homologiquement trivial quelconque. Faisons agir sur $\alpha$ les correspondances apparaissant dans (\ref{eqn:DD}): on a $\Delta_{*}(\alpha)=\alpha$, $(x\times X)_{*}(\alpha)=(l\times h)_{*}(\alpha)=0$ par raison de dimension, et $m'{Z'}_{*}(\alpha)={\tilde i}_{*}\circ {\tilde {Z'}}_{*}(\alpha)$. On obtient
$$mm'\alpha={\tilde i}_{*}\left(\tilde Z'_{*}(\alpha)\right).$$
Comme $\tilde Z'_{*}(\alpha)$ est un diviseur homologue \`a 0 sur $\tilde T$, il est alg\'ebriquement \'equivalent \`a 0 sur $\tilde T$.
Donc $mm'\alpha$ est aussi alg\'ebriquement \'equivalent \`a 0, \emph{i.e.} $\alpha$ est de torsion dans $\Griff^{3}(X)$.
\end{proof}

La combinaison du corollaire \ref{cor:SansTorsion} et de la proposition \ref{prop:DeTorsion} implique que $\Griff^{3}(X)=0$. La partie (1) du th\'eor\`eme \ref{main} est d\'emontr\'ee. \qed

\section{Preuve de l'isomorphisme d'Abel--Jacobi}\label{sect:IsomAJ}
Cette section est consacr\'ee \`a la preuve du th\'eor\`eme \ref{main} (2) qui dit que l'application d'Abel--Jacobi $$\Phi: \CH^{3}(X)_{\alg}\to J^{5}(X)$$ est un isomorphisme.

Rappelons que $\CH^{3}(X)_{\alg}$ est un groupe divisible par le lemme standard suivant\,:
\begin{lem}[{\cite[P.10 Lemma 1.3]{MR2723320}}]\label{lemma:divisible}
Le groupe de Chow des cycles alg\'ebriquement triviaux  est divisible.
\end{lem}

On d\'emontre d'abord que $\Phi$ est une `isog\'enie', ou, dans le langage de \cite{MR714776}, que $\CH^{3}(X)_{\alg}$ est \emph{faiblement repr\'esentable}. 

\begin{prop}\label{prop:Isogenie}
L'application d'Abel--Jacobi $\Phi: \CH^{3}(X)_{\alg}\to J^{5}(X)$ est surjective avec un noyau fini.
\end{prop}
\begin{proof}
On utilisera la d\'ecomposition raffin\'ee de la diagonale (\ref{eqn:DD}) du corollaire \ref{cor:DDcubique5}. Comme dans la d\'emonstration de la proposition \ref{prop:DeTorsion}, on suppose que $T$ est de dimension pure 3 et on choisit une d\'esingularisation  $\tilde T\to T$ \'equip\'ee d'un cycle $\tilde Z'\in \CH^{3}(X\times \tilde T)$ dont l'image dans $\CH_{5}(X\times T)$ est un multiple $m'Z'$ de $Z'$. On note $\tilde i: \tilde T\to X$ le compos\'e de la d\'esingularisation et l'inclusion. Notons que $\tilde T$ n'est pas forc\'ement connexe et on \'ecrit $\Pic^{0}(\tilde T)=\prod \Pic^{0}(\tilde T_{i})$ et $\Alb(\tilde T)=\prod \Alb(\tilde T_{i})$, o\`u $\tilde T=\coprod_{i} \tilde T_{i}$ est la d\'ecomposition en composantes connexes.

Montrons la surjectivit\'e de $\Phi$. Par la compatibilit\'e entre les applications d'Abel--Jacobi et les actions des correspondances (\emph{cf.} \cite[\S 12.2]{MR1988456}), on a le diagramme commutatif suivant\,:
\begin{equation}\label{diag:Surj}
\xymatrix{
 \CH^{3}(X)_{\alg} \ar[r]^-{\Phi} \ar[d]^{{\tilde i}^{*}} \ar@{->>}@/_3pc/[dd]_{\times mm'}& J^{5}(X) \ar[d]^{{\tilde i}^{*}} \ar@{->>}@/^3pc/[dd]^{\times mm'}\\
 \CH^{3}(\tilde T)_{\alg} \ar@{->>}[r]^{\alb} \ar[d]^{{\tilde {Z'}}^{*}}& \Alb(\tilde T)\ar@{->>}[d]^{[\tilde {Z'}]^{*}}\\
 \CH^{3}(X)_{\alg} \ar[r]^-{\Phi}& J^{5}(X)
}
\end{equation}
o\`u les fl\`eches \`a droite sont induites par les morphismes indiqu\'es sur les structures de Hodge. La d\'ecomposition (\ref{eqn:DD}) implique que les deux compos\'ees des fl\`eches verticales dans (\ref{diag:Surj}) sont toutes la multiplication par $mm'$, car les deux premiers termes de (\ref{eqn:DD}) agissent par z\'ero par raison de dimension. Elles sont donc surjectives par la divisibilit\'e. Du coup la fl\`eche verticale du carr\'e inf\'erieur est surjective. Or l'application d'Albanese $\alb$ est aussi surjective, on peut conclure que $\Phi$ est surjective.

Montrons ensuite la finitude du noyau de $\Phi$. \`A nouveau, nous avons un diagramme commutatif
\begin{equation}\label{diag:FiniteKer}
\xymatrix{
 \CH^{3}(X)_{\alg} \ar[r]^-{\Phi} \ar[d]^{{\tilde {Z'}}_{*}} \ar@{->>}@/_3pc/[dd]_{\times mm'}& J^{5}(X) \ar[d]^{[\tilde {Z'}]_{*}} \ar@{->>}@/^3pc/[dd]^{\times mm'}\\
 \CH^{1}(\tilde T)_{\alg} \ar[r]^{\simeq} \ar[d]^{{\tilde i}_{*}}& \Pic^{0}(\tilde T)\ar[d]^{{\tilde i}_{*}}\\
 \CH^{3}(X)_{\alg} \ar[r]^-{\Phi}& J^{5}(X)
}
\end{equation}
De mani\`ere similaire, (\ref{eqn:DD}) implique que les compos\'ees des fl\`eches verticales dans (\ref{diag:FiniteKer}) sont toutes deux la multiplication par $mm'$. Maintenant pour tout $\alpha\in \ker{\Phi}$, par la divisibilit\'e de $\CH^{3}(X)_{\alg}$, rappel\'ee dans le lemme \ref{lemma:divisible}, il existe un $\beta\in \CH^{3}(X)_{\alg}$, tel que $\alpha=mm'\beta$. Donc $\Phi(\beta)\in J^{5}(X)[mm']$. Comme la fl\`eche au milieu de (\ref{diag:FiniteKer}) est un isomorphisme, on trouve que ${\tilde {Z'}}_{*}(\beta)\in \CH^{1}(\tilde T)_{\alg}[mm']$. D'o\`u, $$\alpha=mm'\beta={\tilde i}_{*}\circ{\tilde {Z'}}_{*}(\beta)\in {\tilde i}_{*}\left(\CH^{1}(\tilde T)_{\alg}[mm']\right),$$ qui est un ensemble fini.
\end{proof}

On peut en d\'eduire le th\'eor\`eme \ref{main} (2) pour une hypersurface cubique \emph{g\'en\'erale}\,:
\begin{prop}\label{prop:aj}
Si $X$ est g\'en\'erale, alors l'application d'Abel--Jacobi $\Phi: \CH^3_{alg} \to J^5(X)$ est un isomorphisme de groupes.
\end{prop}
\begin{proof}
Par la proposition \ref{coniveau} (1), il existe une famille de plans contenu dans $X$ param\'etr\'ee par une courbe projective lisse $C$, telle que l'application d'Abel--Jacobi $H^1(C, \ZZ) \to H^5(X, \ZZ)$ est surjective.
On a donc un diagramme commutatif
\begin{equation}
\xymatrix{
 \CH^1(C)_{\alg} \ar[r]^-{\cong} \ar[d]& J(C) \ar@{->>}[d]^{\text{ fib. conn. }}\\
 \CH^3(X)_{\alg} \ar@{->>}[r]^-{\Phi} & J^5(X)
 }
\end{equation}
dont la fl\`eche verticale de droite est surjective \`a fibres connexes par la surjectivit\'e de $H^1(C, \ZZ) \to H^5(X, \ZZ)$. On note $A$ le noyau de cette fl\`eche, qui est une sous-vari\'et\'e ab\'elienne de $J(C)$.

Par la proposition \ref{prop:Isogenie}, $\Phi$ est surjective avec un noyau fini. Comme tout morphisme d'un groupe divisible vers un groupe fini est trivial, l'image de $A$ dans $\ker(\Phi)$ est triviale. Il en r\'esulte que l'image du groupe $\CH^1(C)_{\alg}$ dans $\CH^3(X)_{\alg}$ est isomorphe \`a $J^5(X)$.
Donc on a un scindage de $\Phi$ et $\CH^3(X)_{\alg} \cong J^5(X)\oplus \ker(\Phi)$.
Or le groupe $\CH^3(X)_{\alg}$ est divisible (Lemme \ref{lemma:divisible}), et donc $\ker(\Phi)=0$.
\end{proof}

Pour d\'emontrer le th\'eor\`eme \ref{main} (2) pour une hypersurface cubique lisse \emph{arbitraire}, on a besoin du lemme suivant pour faire un argument de d\'eformation. Voir \cite[\S 20.3]{MR1644323} pour l'application de sp\'ecialisation. 
\begin{lem}[{\emph{cf.}~\cite[Corollaire 1.5]{mboro2017cubic}}]\label{lem:deformation}
Soit $\mathcal{X} \to T$ une famille d'hypersurfaces cubiques lisses de dimension 5 sur une courbe et $\mathcal{X}_{\bar{\eta}}$ la fibre g\'en\'erique g\'eom\'etrique. Alors pour tout point ferm\'e $t \in T$, l'application de sp\'ecialisation $\CH^3(\mathcal{X}_{\bar{\eta}}) \to \CH^3(\mathcal{X}_t)$ est surjective.
\end{lem}
\begin{proof}
La d\'emonstration est quasiment la m\^eme que celle de  \cite[Corollaire 1.5]{mboro2017cubic}, sauf qu'on montre la surjectivit\'e pour le groupe des cycles modulo l'\'equivalence rationnelle au lieu de l'\'equivalence alg\'ebrique, en utilisant un r\'esultat plus fort sur la d\'eformation des courbes d'une vari\'et\'e rationnellement connexe.

Soit $F(\mathcal{X})\to T$ la famille des sch\'emas de Fano des droites de la famille $\mathcal{X}\to T$. 
Alors l'application $\CH_1(F(\mathcal{X}_{t})) \to \CH_2(\mathcal{X}_{t})$ induite par la famille universelle des droites est surjective pour chaque point ferm\'e $t$ de $T$ (\cite[Proposition 1.4]{mboro2017cubic}).  
Donc il suffit de montrer que l'application de sp\'ecialisation $\CH_1(F(\mathcal{X}_{\bar{\eta}})) \to \CH_1(F(\mathcal{X}_t))$ est surjective pour tout $t\in T$. Notons $\mathcal{Y}_{t}:=F(\mathcal{X}_t)$ et $\mathcal{Y}_{\bar\eta}:=F(\mathcal{X}_{\bar\eta})$ respectivement. Ce sont des vari\'et\'es de Fano et donc rationnellement connexes par \cite{MR1189503} et \cite{MR1191735}. 

Soit $f_0: C_0 \to \mathcal{Y}_t$ un morphisme, o\`u $C_{0}$ est une courbe. Gr\^{a}ce \`a la connexit\'e rationnelle de $\mathcal{Y}_{t}$, on peut construire un \emph{peigne} $f: C \to \mathcal{Y}_t$ (voir par exemple \cite[Paragraph 25 et Theorem 27]{MR2011743}) tel que 
\begin{enumerate}
\item La courbe $C=C_0 \cup \bigcup_{i=1}^k C_i$ est une courbe nodale, et chaque point nodal est le seul point d'intersection de $C_{0}$ avec une des courbes $C_1, \ldots, C_k$.
\item Les courbes $C_1, \ldots, C_k$ sont des courbes rationnelles, et le faisceau localement libre $f_i^*T_{\mathcal{Y}_t}$ est ample pour tout $1\leq i\leq k$.
\item La d\'eformation du peigne $f: C \to \mathcal{Y}_t$ est non obstru\'ee.
\end{enumerate}
Par cons\'equent, il existe des courbes $$\mathcal{F}:\mathcal{C}\to \mathcal{Y}_{\bar{\eta}}, ~~\mathcal{F}_i: \mathcal{C}_i \to \mathcal{Y}_{\bar{\eta}}, ~~i=1, \ldots, k,$$ dont les sp\'ecialisations sont des courbes $$f: C \to \mathcal{Y}_t,~~ f_i: C_i \to \mathcal{Y}_t,~~ i=1, \ldots, k.$$
Alors la classe $f_{0,*}(C_0)$ dans $CH_1(\mathcal{Y}_{t})$, qui est $C-\sum_{i=1}^{k}C_{i}$, est dans l'image de l'application de sp\'ecialisation $\CH_1(\mathcal{Y}_{\bar\eta}) \to \CH_1(\mathcal{Y}_t)$.
\end{proof}

\begin{proof}[D\'emonstration du th\'eor\`eme \ref{main}(2)]
La surjectivit\'e de l'application d'Abel--Jacobi $\Phi$ est prouv\'ee dans la proposition \ref{prop:Isogenie}. Il reste \`a montrer son injectivit\'e. Soit $\mathcal{X}\to T$ une famille d'hypersurfaces cubiques lisses de dimension 5 telle que la fibre sur un point $0\in T$ est  isomorphe \`a $X$ et pour tout $t\in T\backslash\{0\}$, l'application d'Abel--Jacobi $\Phi_{t}: \CH^{3}(\mathcal{X}_{t})_{\alg}\to J^{5}(\mathcal{X}_{t})$ est un isomorphisme. L'existence d'une telle famille est garantie par la proposition \ref{prop:aj}.

Pour un cycle  $\alpha\in \ker\left(\Phi\colon\CH^{3}(X)_{\alg}\to J^{5}(X)\right)$ quelconque, par le lemme \ref{lem:deformation}, quitte \`a faire un changement de base de la famille $\mathcal{X}/T$, on peut supposer qu'il existe un cycle $Z\in \CH^{3}(\mathcal{X})$ tel que $Z_0$ est rationnellement \'equivalent \`a $\alpha$ dans $\mathcal{X}_{0}=X$. 
On consid\`ere la fonction normale (\emph{cf.} \cite[D\'efinition 19.5 et Th\'eor\`eme 19.9]{MR1988456}) induite par $Z$, not\'ee $\sigma$, qui est une section de la fibration en jacobiennes interm\'ediaires $\pi: \mathcal{J}^{5}(\mathcal{X}/T)=:\mathcal{J}\to T$. Le fait que $\alpha\in \ker(\Phi)$ dit que $\sigma(0)=0$ dans $\mathcal{J}_{0}$. 

Comme la surjectivit\'e de l'application d'Abel--Jacobi vaut pour toute hypersurface cubique lisse, le morphisme suivant de $T$-sch\'emas est surjectif\,:
\begin{eqnarray*}
 \coprod_{d\in \NN}\Chow^{2,d}(\mathcal{X}/T)\times_{T}\Chow^{2,d}(\mathcal{X}/T)&\to& \mathcal{J}\\
 (z_{1}, z_{2})& \mapsto & \Phi(z_{1}-z_{2}),
\end{eqnarray*}
o\`u $\Chow^{2,d}(\mathcal{X}/T)$ est la vari\'et\'e de Chow du $T$-sch\'ema $\mathcal{X}$ des cycles (effectifs) de dimension 2 et de degr\'e $d$ par rapport \`a certaine polarisation. Par la d\'enombrabilit\'e du nombre des composantes de la vari\'et\'e de Chow et la non-d\'enombrabilit\'e du corps de base $\CC$, il existe un $d$, tel que la restriction \`a la composante $$\mathcal{W}:=\Chow^{2,d}(\mathcal{X}/T)\times_{T}\Chow^{2,d}(\mathcal{X}/T)\to \mathcal{J}$$ est d\'ej\`a un $T$-morphisme surjectif. Quitte \`a faire encore un changement de base qu'on notera toujours par $\mathcal{W}/T$, on peut supposer que le morphisme structural $p: \mathcal{W}\to T$ a une section $r: T\to \mathcal{W}$. On dispose donc d'une famille de 2-cycles $Z'$ dans les fibres de $\mathcal{X}/T$, param\'etr\'ee par $\mathcal{W}/T$, telle que l'application d'Abel--Jacobi 
\begin{eqnarray*}
 AJ_t\colon \mathcal{W}_t &\to& \mathcal{J}_t\\
 w&\mapsto& \Phi(Z'_{w}-Z'_{r(p(w))})
\end{eqnarray*}
 est surjective pour tout $t \in T$.
De plus, on peut supposer que $\mathcal{W}_t$ est irr\'educible pour tout $t \in T$ g\'en\'eral et que le morphisme $\mathcal{W} \to T$ est propre et plat.

On peut donc trouver une courbe $S$ dans $\mathcal{W}$, passant par $r(0)\in \mathcal{W}_{0}$, plate sur $T$, telle que $AJ_{p(s)}(s)=\sigma(s)$ dans $\mathcal{J}_{p(s)}$ pour tout $s\in S$. Par le changement de base $p\colon S\to T$, on obtient une nouvelle famille d'hypersurfaces cubiques lisses $\mathcal{X}_{S}/S$ et une famille de 2-cycles homologiquement triviaux $$\Gamma:=Z\times_{T}S-\left({Z'}\times_{W}S-Z'|_{r(T)}\times_{T} S\right)\in \CH^{3}(\mathcal{X}_{S}),$$ tels que la nouvelle fonction normale est la section nulle de la nouvelle fibration en jacobiennes interm\'ediaires $\mathcal{J}^{5}(\mathcal{X}_{S}/S)\to S$. Gr\^ace \`a la proposition \ref{prop:aj}, pour tout $s\notin p^{-1}(0)$, $\Gamma_{s}$ est rationnellement \'equivalent \`a z\'ero. Par la sp\'ecialisation \cite[\S 20.3]{MR1644323}, $\Gamma_{s}=0$ dans $\CH^{3}\left(\mathcal{X}_{p(s)}\right)$ pour tout $s\in S$. En particulier, 
$$\alpha=\Gamma_{r(0)}=0$$
dans $\CH^{3}(X)$. D'o\`u l'injectivit\'e de $\Phi$. La partie (2) du th\'eor\`eme \ref{main} est \'etablie.
\end{proof}

\section{Preuve du th\'eor\`eme \ref{main2}}\label{sect:SS}

\begin{lem}\label{lem:H1i}
Pour tout $i\geq 2$, le groupe $H^{1}(X, \HH^{i}(\ZZ))$ est nul.
\end{lem}
\begin{proof}
Par le lemme \ref{lemma:SES}, on a une suite exacte courte
$$0\to H^{i}_{\nr}(X, \ZZ)\otimes \QQ/\ZZ\to H^{i}_{\nr}(X, \QQ/\ZZ)\to H^{1}(X, \HH^{i}(\ZZ))_{\tor}\to 0.$$
Donc $H^{1}(X, \HH^{i}(\ZZ))$ est sans torsion par l'annulation de $H^{i}_{\nr}(X, \QQ/\ZZ)$ (\S \ref{subsect:VanishingHnr}, lemmes \ref{nr4} et \ref{nr5}). Il reste \`a voir qu'il est de torsion. On fait agir les deux c\^ot\'es de la d\'ecomposition de la diagonale (\ref{eqn:DD}), modulo l'\'equivalence alg\'ebrique (\emph{cf.} \cite[Appendice A]{CTVoisin2012IHC}), sur $H^{1}(X, \HH^{i}(\ZZ))$ et on trouve que
\begin{itemize}
\item l'action $\Delta_{X, *}$ est l'identit\'e\,;
\item l'action $(x\times X)_{*}$ se factorise par $H^{1}(\{x\}, \HH^{i}_{\{x\}}(\ZZ))$, or le faisceau $\HH^{i}_{\{x\}}(\ZZ)=0$ par \cite{MR0177422}, pour tout $i\geq 1$\,;
\item l'action $(l\times h)_{*}$ se factorise par $H^{1}(l, \HH^{i}_{l}(\ZZ))$, or le faisceau $\HH^{i}_{l}(\ZZ)=0$ par \cite{MR0177422}, pour tout $i\geq 2$\,;
\item l'action $(Z')_{*}$ se factorise par $H^{-1}(\tilde T, \HH^{i-2}_{\tilde T}(\ZZ))=0$, o\`u $\tilde T$ est une r\'esolution des singularit\'es de $T$, comme dans le d\'ebut de la d\'emonstration de la proposition \ref{prop:DeTorsion}.
\end{itemize}
Par cons\'equent, $H^{1}(X, \HH^{i}(\ZZ))$ est annul\'e par la multiplication par $m$.
\end{proof}

\begin{lem}\label{lem:Hii}
Pour tout $0\leq i\leq 5$, l'application classe $$H^{i}(X, \HH^{i}(\ZZ))\cong \CH^{i}(X)/\alg\to H^{2i}(X, \ZZ)\simeq \ZZ$$ est un isomorphisme.
\end{lem}
\begin{proof}
La surjectivit\'e est \'equivalente \`a la conjecture de Hodge enti\`ere, qui vaut pour $X$ par l'existence des droites et des plans dans $X$.  L'injectivit\'e est \'equivalente \`a l'annulation du groupe de Griffiths $\Griff^{i}(X)$, qui est prouv\'ee dans \S \ref{sect:Griff} pour $i=3$ et facile pour les autres $i$ (\emph{cf.} \S \ref{sect:Intro}).
\end{proof}

\begin{proof}[D\'emonstration du th\'eor\`eme \ref{main2}]
On regarde la suite spectrale de Leray
$$E_{2}^{p,q}=H^{p}(X, \HH^{q}(\ZZ))\Longrightarrow H^{p+q}(X, \ZZ).$$
En tenant compte les annulations 
\begin{itemize}
\item $H^{p}(X, \HH^{q}(\ZZ))=0$ pour tout $q>\dim X=5$, par \cite{MR0177422}\,;
\item $H^{p}(X, \HH^{q}(\ZZ))=0$ pour tout $p>q$, par \cite[0.3]{MR0412191}\,;
\item $H^{0}(X, \HH^{q}(\ZZ))=0$ pour tout $q>0$, par \cite[Proposition 3.3(i)]{CTVoisin2012IHC}\,;
\item $H^{1}(X, \HH^{q}(\ZZ))=0$ pour tout $q>1$, par le lemme \ref{lem:H1i},
\end{itemize}
on voit que dans cette suite spectrale,  il n'y qu'une seule fl\`eche $d_{2}\colon H^{2}(X, \HH^{5}(\ZZ))\to H^{4}(X, \HH^{4}(\ZZ))$ qui pourrait \^etre non nulle. Montrons que $H^{2}(X, \HH^{5}(\ZZ))$ est un groupe de torsion par un argument similaire \`a celui du lemme \ref{lem:H1i}: on fait agir les deux c\^ot\'es de (\ref{eqn:DD}) sur le groupe $H^{2}(X, \HH^{5}(\ZZ))$:
\begin{itemize}
\item l'action $\Delta_{X}^{*}$ est l'identit\'e\,;
\item l'action $(x\times X)^{*}$ se factorise par $H^{-3}(\{x\}, \HH^{0}_{\{x\}}(\ZZ))=0$\,;
\item l'action $(l\times h)^{*}$ se factorise par $H^{-2}(l, \HH^{1}_{l}(\ZZ))=0$\,;
\item l'action $(Z')^{*}$ se factorise par $H^{2}(\tilde T, \HH^{5}_{\tilde T}(\ZZ))$, o\`u $\tilde T$ est une r\'esolution des singularit\'es de $T$, comme dans le d\'ebut de la d\'emonstration de la proposition \ref{prop:DeTorsion}. Or le faisceau $\HH^{5}_{\tilde T}(\ZZ)=0$ par \cite{MR0177422}, car $\dim \tilde T=3<5$.
\end{itemize}
Par cons\'equent, $H^{2}(X, \HH^{5}(\ZZ))$ est bien de torsion et donc n'admet pas de morphisme non nul vers le groupe ab\'elien libre $H^{4}(X, \HH^{4}(\ZZ))$. Par cons\'equent, la suite spectrale est d\'eg\'en\'er\'ee en $E_{2}$. La partie (3) du th\'eor\`eme \ref{main2} est d\'emontr\'ee (\emph{cf.} Proposition \ref{prop:TwoSS}).

La d\'eg\'en\'erescence de la suite spectrale \'etant prouv\'ee, le lemme \ref{lem:Hii}, avec les annulations ci-dessus, implique la partie (1) du th\'eor\`eme \ref{main2}.

Finalement, la partie (2) du th\'eor\`eme \ref{main2} d\'ecoule de la partie (1), en particulier du fait que tous les termes de la suite spectrale pour le faisceau $\ZZ$ sont sans torsion, et du lemme \ref{lemma:SES}.
\end{proof}

\bibliographystyle{plain}
\bibliography{biblio.bib}

\begin{thebibliography}{10}

\bibitem{MR2115000}
Yves Andr{\'e}.
\newblock {\em Une introduction aux motifs (motifs purs, motifs mixtes,
  p{\'e}riodes)}, volume~17 of {\em Panoramas et Synth{\`e}ses [Panoramas and
  Syntheses]}.
\newblock Soci{\'e}t{\'e} Math{\'e}matique de France, Paris, 2004.

\bibitem{MR0177422}
Aldo Andreotti et Theodore Frankel.
\newblock The {L}efschetz theorem on hyperplane sections.
\newblock {\em Ann. of Math. (2)}, 69:713--717, 1959.

\bibitem{MR2011743}
Carolina Araujo et J\'anos Koll\'ar.
\newblock Rational curves on varieties.
\newblock In {\em Higher dimensional varieties and rational points ({B}udapest,
  2001)}, volume~12 of {\em Bolyai Soc. Math. Stud.}, pages 13--68. Springer,
  Berlin, 2003.

\bibitem{MR2723320}
Spencer Bloch.
\newblock {\em Lectures on algebraic cycles}, volume~16 of {\em New
  Mathematical Monographs}.
\newblock Cambridge University Press, Cambridge, second edition, 2010.

\bibitem{MR0412191}
Spencer Bloch et Arthur Ogus.
\newblock Gersten's conjecture and the homology of schemes.
\newblock {\em Ann. Sci. {\'E}cole Norm. Sup. (4)}, 7:181--201 (1975), 1974.

\bibitem{MR714776}
Spencer Bloch et Vasudevan Srinivas.
\newblock Remarks on correspondences and algebraic cycles.
\newblock {\em Amer. J. Math.}, 105(5):1235--1253, 1983.

\bibitem{MR1191735}
Fr\'{e}d\'{e}ric Campana.
\newblock Connexit\'{e} rationnelle des vari\'{e}t\'{e}s de {F}ano.
\newblock {\em Ann. Sci. \'{E}cole Norm. Sup. (4)}, 25(5):539--545, 1992.


\bibitem{collino1986AJcubic5fold}
Alberto Collino.
\newblock The {A}bel-{J}acobi isomorphism for the cubic fivefold.
\newblock {\em Pacific J. Math.}, 122(1):43--55, 1986.

\bibitem{MR1327280}
Jean-Louis Colliot-Th\'el\`ene.
\newblock Birational invariants, purity and the {G}ersten conjecture.
\newblock In {\em {$K$}-theory and algebraic geometry: connections with
  quadratic forms and division algebras ({S}anta {B}arbara, {CA}, 1992)},
  volume~58 of {\em Proc. Sympos. Pure Math.}, pages 1--64. Amer. Math. Soc.,
  Providence, RI, 1995.

\bibitem{MR3404380}
Jean-Louis Colliot-Th\'el\`ene.
\newblock Descente galoisienne sur le second groupe de {C}how: mise au point et
  applications.
\newblock {\em Doc. Math.}, (Extra vol.: Alexander S. Merkurjev's sixtieth
  birthday):195--220, 2015.

\bibitem{MR999316}
Jean-Louis Colliot-Th\'el\`ene et Manuel Ojanguren.
\newblock Vari\'et\'es unirationnelles non rationnelles: au-del\`a de l'exemple
  d'{A}rtin et {M}umford.
\newblock {\em Invent. Math.}, 97(1):141--158, 1989.

\bibitem{CTVoisin2012IHC}
Jean-Louis Colliot-Th{\'e}l{\`e}ne et Claire Voisin.
\newblock Cohomologie non ramifi{\'e}e et conjecture de {H}odge enti{\`e}re.
\newblock {\em Duke Math. J.}, 161(5):735--801, 2012.

\bibitem{MR1644323}
William Fulton.
\newblock {\em Intersection theory}, volume~2 of {\em Ergebnisse der Mathematik
  und ihrer Grenzgebiete. 3. Folge. A Series of Modern Surveys in Mathematics
  [Results in Mathematics and Related Areas. 3rd Series. A Series of Modern
  Surveys in Mathematics]}.
\newblock Springer-Verlag, Berlin, second edition, 1998.

\bibitem{MR1637963}
Bruno Kahn, Markus Rost, et R.~Sujatha.
\newblock Unramified cohomology of quadrics. {I}.
\newblock {\em Amer. J. Math.}, 120(4):841--891, 1998.

\bibitem{kahn2001unramifiedquadric2}
Bruno Kahn et R.~Sujatha.
\newblock Unramified cohomology of quadrics. {II}.
\newblock {\em Duke Math. J.}, 106(3):449--484, 2001.

\bibitem{MR1189503}
J\'{a}nos Koll\'{a}r, Yoichi Miyaoka, et Shigefumi Mori.
\newblock Rational connectedness and boundedness of {F}ano manifolds.
\newblock {\em J. Differential Geom.}, 36(3):765--779, 1992.


\bibitem{MR1669995}
Robert Laterveer.
\newblock Algebraic varieties with small {C}how groups.
\newblock {\em J. Math. Kyoto Univ.}, 38(4):673--694, 1998.

\bibitem{mboro2017cubic}
Ren{\'e} Mboro.
\newblock Remarks on the {$CH_2$} of cubic hypersurfaces.
\newblock {\em Geometriae Dedicata, https://doi.org/10.1007/s10711-018-0355-0},
  2018.

\bibitem{MR675529}
A.~S. Merkurjev et A.~A. Suslin.
\newblock {$K$}-cohomology of {S}everi-{B}rauer varieties and the norm residue
  homomorphism.
\newblock {\em Izv. Akad. Nauk SSSR Ser. Mat.}, 46(5):1011--1046, 1135--1136,
  1982.

\bibitem{MR1062517}
A.~S. Merkurjev et A.~A. Suslin.
\newblock Norm residue homomorphism of degree three.
\newblock {\em Izv. Akad. Nauk SSSR Ser. Mat.}, 54(2):339--356, 1990.

\bibitem{MR777590}
Jacob~P. Murre.
\newblock Un r\'esultat en th\'eorie des cycles alg\'ebriques de codimension
  deux.
\newblock {\em C. R. Acad. Sci. Paris S\'er. I Math.}, 296(23):981--984, 1983.

\bibitem{MR805336}
Jacob~P. Murre.
\newblock Applications of algebraic {$K$}-theory to the theory of algebraic
  cycles.
\newblock In {\em Algebraic geometry, {S}itges ({B}arcelona), 1983}, volume
  1124 of {\em Lecture Notes in Math.}, pages 216--261. Springer, Berlin, 1985.

\bibitem{Murre}
Jacob~P. Murre.
\newblock On a conjectural filtration on the {C}how groups of an algebraic
  variety. {I}. {T}he general conjectures and some examples.
\newblock {\em Indag. Math. (N.S.)}, 4(2):177--188, 1993.

\bibitem{MR1283872}
Kapil~H. Paranjape.
\newblock Cohomological and cycle-theoretic connectivity.
\newblock {\em Ann. of Math. (2)}, 139(3):641--660, 1994.

\bibitem{MR1424593}
Kapil~H. Paranjape.
\newblock Some spectral sequences for filtered complexes and applications.
\newblock {\em J. Algebra}, 186(3):793--806, 1996.

\bibitem{RostNonpub}
Markus Rost.
\newblock On {H}ilbert {S}atz 90 for {K}3 for degree-two extensions.
\newblock {\em non-publi{\'e},
  http://www.math.uni-bielefeld.de/~rost/K3-86.html}.

\bibitem{MR1324577}
Jean-Pierre Serre.
\newblock {\em Cohomologie galoisienne}, volume~5 of {\em Lecture Notes in
  Mathematics}.
\newblock Springer-Verlag, Berlin, fifth edition, 1994.

\bibitem{MR2220090}
Andrei Suslin et Seva Joukhovitski.
\newblock Norm varieties.
\newblock {\em J. Pure Appl. Algebra}, 206(1-2):245--276, 2006.

\bibitem{MR2031199}
Vladimir Voevodsky.
\newblock Motivic cohomology with {${\bf Z}/2$}-coefficients.
\newblock {\em Publ. Math. Inst. Hautes \'Etudes Sci.}, (98):59--104, 2003.

\bibitem{MR2811603}
Vladimir Voevodsky.
\newblock On motivic cohomology with {$\bold Z/l$}-coefficients.
\newblock {\em Ann. of Math. (2)}, 174(1):401--438, 2011.

\bibitem{MR1988456}
Claire Voisin.
\newblock {\em Th{\'e}orie de {H}odge et g{\'e}om{\'e}trie alg{\'e}brique
  complexe}, volume~10 of {\em Cours Sp{\'e}cialis{\'e}s [Specialized
  Courses]}.
\newblock Soci{\'e}t{\'e} Math{\'e}matique de France, Paris, 2002.

\bibitem{MR2987669}
Claire Voisin.
\newblock Degree 4 unramified cohomology with finite coefficients and torsion
  codimension 3 cycles.
\newblock In {\em Geometry and arithmetic}, EMS Ser. Congr. Rep., pages
  347--368. Eur. Math. Soc., Z\"urich, 2012.

\bibitem{MR3186044}
Claire Voisin.
\newblock {\em Chow rings, decomposition of the diagonal, and the topology of
  families}, volume 187 of {\em Annals of Mathematics Studies}.
\newblock Princeton University Press, Princeton, NJ, 2014.

\bibitem{MR3586396}
Claire Voisin.
\newblock Integral {H}odge classes, decompositions of the diagonal, and
  rationality questions.
\newblock In {\em Trends in contemporary mathematics}, volume~8 of {\em
  Springer INdAM Ser.}, pages 137--149. Springer, Cham, 2014.

\end{thebibliography}

\end{document}